\documentclass[11pt]{amsart}
\newtheorem{theorem}{Theorem}[section]

\newtheorem{corollary}[theorem]{Corollary}
\newtheorem{lemma}[theorem]{Lemma}

\newtheorem{proposition}[theorem]{Proposition}
\newtheorem{remark}[theorem]{Remark}
\newtheorem{definition}[theorem]{Definition}

\usepackage[vmargin=2cm, hmargin=2cm]{geometry}

\usepackage{fourier}

\usepackage{hyperref}
\usepackage{enumerate}
\usepackage{amsmath}
\usepackage{amssymb}

\def\CC{{\mathbb{C}}}

\def\soc{{\rm Soc}}

\begin{document}

\title[Minus partial order and linear preservers]{ Minus partial order and linear preservers}

\author{M. Burgos}

\address{Campus de Jerez, Facultad de Ciencias Sociales y de la Comunicaci\'{o}n Av. de la Universidad s/n, 11405 Jerez, C\'{a}diz, Spain}

\email{maria.burgos@uca.es}

\author{A. C. M\'{a}rquez-Garc\'{i}a}

\address{ Departamento \'{A}lgebra y An\'{a}lisis Matem\'{a}tico,
Universidad de Almer\'{i}a, 04120 Almer\'{i}a, Spain}

\email{acmarquez@ual.es}

\author{A. Morales-Campoy}

\address{Departamento de \'{A}lgebra y An\'{a}lisis Matem\'{a}tico,
Universidad de Almer\'{i}a, 04120 Almer\'{i}a, Spain}

\email{amorales@ual.es}

\thanks{Authors partially supported by the Spanish Ministry of Economy and Competitiveness project no. MTM2014-58984-P and Junta de Andaluc\'{\i}a grants FQM375, FQM194. The second author is also supported by a Plan Propio de Investigaci\'{o}n grant from University of Almer\'{i}a.}

\begin{abstract} In this paper we investigate the properties of minus partial order in unital rings. We generalize several results well known for matrices and bounded linear operators on Banach spaces. We also study linear maps preserving the minus partial order in unital semisimple Banach algebras with essential socle.
\end{abstract}


\date{}

\maketitle
 \thispagestyle{empty}

\section{Introduction}
Let $A$ be a ring. Recall that an element $a\in A$ is \emph{regular} if there is $b\in A$ such that $aba=a$.  We write $A^\wedge$ for the set of regular elements in $A$.

For a regular element $a\in A$, the set $$ G_1(a)=\{x\in A\colon axa=a\}$$ consists of all $\{1\}$-\emph{inverses} or \emph{inner inverses} of $a$. Notice that if $x$ is a  $\{1\}$-inverse of $a$, then $ax$ and $xa$ are idempotents. A  $\{1,2\}$-\emph{inverse} or \emph{generalized inverse} of $a$, is a $\{1\}$-inverse of $a$ that is a solution of the equation $xax=x$, that is, it is an element $b\in A$ such that $aba=a$ and $bab=b$. Let us denote by $ G_2(a)$ the set of generalized inverses of a regular element $a\in A$.

Note that the condition $x\in G_1(a)$ ensures the existence of a generalized inverse of $a$: in such case, $b=xax$ fulfills the previous identities.

Let $A$ be a Banach algebra. For an element $a$ in $A$, let us consider the left and right multiplication operators $L_a
:x\mapsto ax$ and $R_a:x\mapsto xa$, respectively. If $a$ is regular, then so
are $L_a$ and $R_a$, and thus their ranges $aA=L_a(A)$ and $Aa=R_a(A)$ are
both closed.

Even though regularity can be defined in general Banach algebras, this notion has been mostly studied in  C*-algebras. Harte and Mbekhta proved  in \cite{HarMb92} that an element $a$ in a unital C*-algebra $A$ is regular if and only if $aA$ is closed. Given $a$ and $b$ in a C$^*$-algebra $A$, we shall say that $b$ is a \emph{Moore-Penrose inverse} of $a$ if $b$ is a generalized inverse of $a$ and $ab$ and $ba$ are selfadjoint. It is known that every regular element $a$ in $A$ has a unique Moore-Penrose inverse that will be denoted by $a^\dag$ (\cite{HarMb92}).

Since the 80's, many authors have focused on the study of some partial orders defined in abstract structures, such as semigroups, rings of matrices and, more specifically, algebras (see \cite{Drazin78}, \cite{Hart80}, \cite{Mi86}, \cite{MiBhiMa}).

Let $M_n(\CC)$ be the algebra of all $n\times n$ complex matrices. The \emph{star partial order} on $M_n(\CC)$ was introduced by Drazin  in \cite{Drazin78}, as follows:
$$ A\leq_{*}B \qquad \mbox{if and only if}\qquad A^*A=A^*B \,\,  \mbox{and } \,\, AA^*=BA^*,$$
where as usual $A^*$ denotes the conjugate transpose of $A$. It was proved that $ A\leq_{*}B$ if and only if $A^\dag A=A^\dag B $ and $ AA^\dag=BA^\dag .$

Hartwig \cite{Hart80} introduced the \emph{rank substractivity order} on $M_n(\CC)$:
$$ A\leq ^{-} B \qquad \mbox{if and only if}\qquad \textrm{rank}(B-A)= \textrm{rank}(B)- \textrm{rank}(A).$$ He proved that
$$ A\leq^{-}B \qquad \mbox{if and only if}\qquad A^-A=A^-B \,\,  \mbox{and }\, \, AA^-=BA^-,$$ where $A^-$ denotes a $\{1\}$-inverse of $A$. This partial order is usually named the \emph{minus partial order}.

Later, Mitra introduced in \cite{Mi91} the \emph{space pre-order} on $M_n(\CC)$:
$$M\leq_s N \quad \mbox{if and only if}\quad \mathcal{C}(M)\subseteq \mathcal{C}(N)\quad\mbox{ and }\quad \mathcal{C}(N^*)\subseteq \mathcal{C}(M^*),$$ where $\mathcal{C}(M)$ denotes the column space of the matrix $M$.

In \cite{Dragans}, Raki\'c and Djordjevi\'c extend the definition of space pre-order to the class of bounded linear operators on Banach spaces, and generalize some well known properties of this partial order to the new setting.

Let $H$ be an infinite-dimensional complex Hilbert space, and $B(H)$ the C$^*$-algebra af all bounded linear operators on $H$. Having into account that an operator in $B(H)$ is regular if and only if it has closed range, \v{S}emrl (\cite{Semrl10}) extended the minus partial order from  $M_n(\CC)$ to $B(H)$, finding and appropriate equivalent definition of the minus partial order on  $M_n(\CC)$ which does not involve $\{1\}$-inverses: for $A,B\in B(H)$, $A\preceq B$ if and only if there exists idempotent operators $P,Q\in B(H)$ such that
$$R(P)=\overline{R(A)},\quad N(A)=N(Q), \quad PA=PB, \quad AQ=BQ.$$
\v{S}emrl proved that the relation  $\preceq $ is a partial order in $B(H)$ extending the minus partial order of matrices. Finally, Djordjevi\'c, Raki\'c and Marovt (\cite{DjoRaMa13}) generalized \v{S}emrl's definition to the environment of Rickart rings (see Definition \ref{orders}) and generalized some well known results. Recall that a ring $A$ is a \emph{Rickart ring} if the left and right annihilator of any element are generated by idempotent elements.

Several order relations received similar treatment (see \cite{Marovt15}, \cite{Dragans15} and the references therein).

During the last three decades several results involving linear preservers of order relations have been published. In 1993, Ovchinnikov showed in \cite{Ovchi} that every bijective map $\phi$ defined on the set $B(H)^\bullet$ of idempotents operators on a complex Hilbert space, satisfying that $\phi(P)\leq \phi(Q)$ if and only if $P\leq Q$ can be expressed either as $\phi(P)=APA^{-1}$ for every $P\in B(H)^\bullet$, or as $\phi(P)=AP^*A^{-1}$ for every $P\in B(H)^\bullet$, where $A$ is a linear or conjugate-linear bijection on $H$. Later, many results concerning order preserving maps in matrix algebras appeared in the literature (the reader is referred to  \cite{ Guterman, Legisa, Semrl03}).
In   \cite{Semrl10},  \v{S}emrl  studied (non necessarily linear) bijective maps preserving the minus partial order. For an infinite-dimensional complex Hilbert space $H$, a mapping  $\phi:B(H)\to B(H)$  preserves the minus order if  $ A\preceq B$ implies that  $ \phi(A)\preceq\phi(B)$.
The map  $\phi:B(H)\to B(H)$  preserves the minus order in both directions whenever  $ A\preceq B$ if and only if $ \phi(A)\preceq \phi(B)$. He proved that a bijective map  $\phi:B(H)\to B(H)$  preserving the minus order in both directions has the form $\phi(A)=TAS$ or $\phi(A)=TA^*S$, for some invertible operators $T$ and $S$ (both linear in the first case and both conjugate linear in the second one).
The star-type orders preservers have been studied in similar conditions (\cite{Bohata, DoGuMa, DoGuMa2}).

Let $A$ be a unital ring. For every subset $M$ of $A$, the \emph{right annihilator} of $M$ is denoted by $$\textrm{ann}_r (M)=\{x\in A\colon mx=0 \mbox{ for all }m\in M\}.$$
Similarly, the \emph{left annihilator} of $M$ is given by $$\textrm{ann}_l (M)=\{x\in A\colon xm=0 \mbox{ for all }m\in M\}.$$
For an element $a\in A$, it is used to write  $\textrm{ann}_r (a)= \textrm{ann}_r (\{a\})$ and $\textrm{ann}_l (a)= \textrm{ann}_l (\{a\})$.
The set of idempotent elements of $A$ is denoted by $A^ \bullet$.

We will adopt the definition from \cite{DjoRaMa13}:

\begin{definition}\label{orders} We say that $a\leq^- b$ if there exist $p,q\in A^\bullet $ such that  $\textrm{ann}_l(a)=\textrm{ann}_l(p)$, $\textrm{ann}_r(a)=\textrm{ann}_r(q)$, $pa=pb$ and $aq=bq$.

\end{definition}

The paper is organized as follows.  Section 2 is devoted to the study of the relation ''$\leq^{-}$'' in unital rings. We prove some algebraic properties of this relation and show that ''$\leq^{-}$'' defines a partial order on the set of all regular elements of a semiprime ring (Corollary \ref{minuspartord}). The space pre-order in general unital rings (Definition \ref{space}) it is also considered.
 We characterize the maximal elements of the set of regular elements with respect to this relation in a unital prime ring (Proposition \ref{max}). We also determine the minimal elements  of the set of regular elements with respect to this relation in a unital semisimple Banach algebra with essential socle (Proposition \ref{invert}). This will be the key tool in order to obtain the main result in Section 3. When $A$ and $B$ are unital semisimple Banach algebras with essential socle, we prove in Theorem \ref{main1} that every bijective linear mapping $\Phi:A \to B$ such that $\Phi(A^\wedge)= B^\wedge$, and $a\leq^-b \Leftrightarrow \Phi(a)\leq^- \Phi(b)$, for every $a,b\in A^\wedge$ is a Jordan isomorphism multiplied by an invertible element. The condition $\Phi(A^\wedge)= B^\wedge$ can be removed either when $B=B(X)$ for a complex Banach space $X$ (Theorem \ref{op}) or $B$ is a prime $C^*$-algebra (Theorem \ref{mainminus}). We also consider briefly linear mappings preserving the minus partial order and the space pre-order in just one direction.

\section{The minus partial order}
\subsection{Algebraic properties}

In the next proposition we collect some properties of the relation '' $ \leq^{-}$ '' in a unital ring that we will need in the sequel.
\begin{proposition}\label{reg1}Let $A$ be a unital ring. The following assertions hold:
\begin{enumerate}
\item  If $a\in A^\wedge$, then $a\leq^- b$ if and only if there exists $a^-\in G_1(a)$ such that $a^-a=a^-b$ and $aa^-=ba^-$.
\item If $b\in A^\wedge$ and  $a\in A$ satisfy that $a\leq^-b$, then $a\in A^\wedge$ and $G_1(b)\subset G_1(a)$.
\item If  $a, b\in A^\wedge$, then $a\leq^- b$ if and only if there exists $b^- \in G_1(b)$ such that $a=ab^- b= bb^- a=ab^-a$.
\item For every invertible element $u\in A$, $$a\leq^-b \Leftrightarrow ua\leq^- ub$$ and $$a\leq^-b \Leftrightarrow au\leq^- bu,$$ for every $a,b\in A$.
\item If $p\in A^\bullet$ and $a\leq^{-}p$ then $a\in A^\bullet$ and $a=ap=pa$.
\end{enumerate}
\end{proposition}
\begin{proof}


\emph{(1)} Let $a,b\in A$. If $a\leq^- b$ and $p, q$ are the idempotents appearing in Definition \ref{orders}, since $(1-p)p=0$, then $(1-p)a=0$ and, consequently, $a=pa$. Similarly, $a=aq$. Note also that, if $a\in A$ is a regular element, then $a\leq^- b$ if and only if there exists $a^-\in G_1(a)$ such that $a^-a=a^-b$ and $aa^-=ba^-$. Indeed, let $b\in A$ such that $a\leq^- b$, and let $p, q\in A^\bullet $ as in Definition \ref{orders}. Take $x\in G_1(a)$. Since  $a=pa$ and $a=aq$ it is clear that $a^{-}:=qxp$ is an inner inverse of $a$. Moreover $aa^{-}=aqxp=bqxp=ba^{-}$. Analogously, it can be checked that $a^{-}a=a^{-}b$. Notice that we can actually choose $a^+G_2(a)$ satisfying $a^+a=a^+b$ and $aa^+=ba^+$ by putting $a^+=a^-aa^-$.

Reciprocally, suppose that $a\in A^{\wedge}$ and  $b\in A$ satisfy that $aa^-=ba^-$ and $a^-a=a^-b$. Then $p=aa^-$ and $q=a^-a$ are idempotents,  $\textrm{ann}_l(a)=\textrm{ann}_l(p)$, $\textrm{ann}_r(a)=\textrm{ann}_r(q)$, $a=pa=aa^-a=aa^-b=pb$ and $a=aq=aa^-a=ba^-a=bq$. This shows that
$a\leq ^{-} b$.

\emph{(2)}  Let $b\in A^\wedge$ and  $a\in A$ such that $a\leq^-b$. There exist $p,q\in A^\bullet$ verifying $a=pa=pb$ and $a=aq=bq$. For an arbitrary $b^-\in G_1(b)$, multiplying the first identity by $b^-b$ on the right we obtain $ab^-b=pbb^-b=pb=a$. Multiplying now by $q$ on the right it yields $ab^-bq=aq$, that is, $ab^-a=a$ and, consequently, $b^-\in G_1(a)$.

\emph{(3)} By looking at the proof of \cite[Lemma 2 (a)]{LebPaTh13}, it can be seen that the same statement holds only assuming that the elements $a$ and $b$ are regular. In other words, the hypothesis of $R$ being regular can be relaxed to $a,b$ regular. Notice also that if $a\leq^- b$, then  $a=ab^- b= bb^- a=ab^-a,$ for every $b^- \in G_1(b)$.

\emph{(4)} If $a\leq^-b $, there exist $p,q\in A^\bullet $ such that  $\textrm{ann}_l(a)=\textrm{ann}_l(p)$, $\textrm{ann}_r(a)=\textrm{ann}_r(q)$, $pa=pb$ and $aq=bq$. Let $p_u=upu^{-1}$. Then $p_u\in A^\bullet$, $\textrm{ann}_l(ua)=\textrm{ann}_l(p_u)$, $p_u ua=p_u ub$, and $uaq=ubq$. This shows that $a\leq^-b \Leftrightarrow ua\leq^- ub$.
Similarly, it can be proved that for any invertible element $u\in A$, $a\leq^-b \Leftrightarrow au\leq^- bu$, for every $a,b\in A$.

\emph{(5)} We know from \emph{(2)}  that $a\in A^\wedge$ and $G_1(p)\subset G_1(a)$. From \emph{(3)} $a=ap=pa=apa$. In particular $a^2=apa=a$.

\end{proof}
There are many characterizations of the minus partial order. In \cite{Mi86} it is proved that for complex matrices $M$ and $N$ of the same order, $M\leq^{-} N$ if and only if $ G_1(N)\subseteq  G_1(M)$. This result was latter extended to the setting of regular rings in \cite{BlJaPSr09}. In the next proposition we show that the relation ''$\leq ^{-}$'' is equivalent to the inclusion of the set of $\{1\}$-inverses for regular elements on a unital semiprime ring.
For a regular element $a\in A$, we define $D_1(a)=\{x-y \colon x,y\in G_1(a)\}$.
\begin{lemma}\label{diff}
Let $A$ be a unital ring and $a\in A^\wedge$. Then
$$D_1(a)=\{x\in A\colon axa=0\}.$$
\end{lemma}


\begin{proof} Pick $x\in D_1(a)$. Then $x=a^- -a^=$ for some $a^-,a^=\in G_1(a)$. Hence $$axa=a(a^- - a^=)a=a-a=0.$$
For the reciprocal inclusion, suppose that $axa=0$ and take $a^- \in G_1(a)$. As
$a(a^{-}-x)a=aa^{-}a -axa=a$, it is clear that $a^{-}, a^{-}-x \in G_1(a)$ and, consequently, $x= a^{-}- (a^{-}-x) \in D_1(a)$, as desired.
\end{proof}
\begin{proposition}\label{minusreg} Let $A$ be a unital semiprime ring and $a,b\in A^\wedge$.
The following assertions are equivalent:
\begin{enumerate}
\item $a\leq^- b$,
\item $G_1(b)\subset G_1(a)$,
\item   $G_1(b)\cap G_1(a)\neq \emptyset$ and $D_1(b)\subset D_1(a)$.
\end{enumerate}

\end{proposition}

\begin{proof} Let $a,b\in A^\wedge$. We know from Proposition \ref{reg1} that $G_1(b)\subset G_1(a)$ whenever $a\leq^- b$, and hence \emph{$(1)\Rightarrow (2)$}.

It is clear that \emph{(2)}$\Rightarrow$\emph{(3)}.

Finally suppose that \emph{(3)} holds and let $b^-\in G_1(b)\cap G_1(a)$. Since $a$ and $b$ are regular,  in order to prove that $a\leq^- b$, it is enough to show that $a=ab^-b=bb^-a$.
Taking into account that $b(1-b^-b)Ab=\{0\}$ for every $b^-\in G_1(b)$ and that  $D_1(b)\subset D_1(a)$, we conclude by Lemma \ref{diff} that $a(1-b^-b)Aa=\{0\}$. Therefore, $$a(1-b^-b)Aa(1-b^-b)=\{0\},$$ which, being $A$ a semiprime algebra, gives $a=ab^-b$. Similar arguments can be applied to get $a=bb^-a$.

\end{proof}

In \cite[Theorem 3.3]{DjoRaMa13} the authors showed that the relation ''$\leq ^{-}$'' is a partial order on a Rickart ring. Also, from \cite[Theorem 3.3]{Dragans}, ''$\leq ^{-}$'' is a partial order on the class of relatively regular operators on Banach spaces. As consequence of the above proposition we generalize this result to the setting of unital semiprime rings.
\begin{corollary}\label{minuspartord}
Let $A$ be a unital semiprime ring. The relation $\leq ^{-}$ is a partial order on $A^\wedge$.
\end{corollary}


\begin{proof} Reflexivity and transitivity of the relation ''$\leq ^{-}$'' follow directly from Proposition \ref{minusreg}. In order to prove the anti-symmetry, take $a,b \in A^\wedge$ with $a\leq ^{-}b$ and $b\leq ^{-}a$. There exists $a^{-}\in G_1(a)$ and $b^{-}\in G_1(b)$ such that
$$aa^{-}=ba^{-},\, a^{-}a=a^{-}b,$$
$$bb^{-}=ab^{-},\, b^{-}b=b^{-}a.$$
Since $G_1(a)= G_1(b)$, it follows that $b^{-}\in G_1(a)$. That is
$$a=ab^{-}a=ab^{-}b=bb^{-}b=b.$$
\end{proof}

\begin{definition}\label{space}
Let $A$ be a ring. We say that $a\leq_s b$ if  $aA\subset bA$ and $Aa\subset Ab$.
\end{definition}
This definition is analogous to the definition of the space pre-order on complex matrices introduced by Mitra in \cite{Mi91}. Recall that $M\leq_s N$ if $\mathcal{C}(M)\subseteq \mathcal{C}(N)$ and $\mathcal{C}(N^*)\subseteq \mathcal{C}(M^*)$, where $\mathcal{C}(M)$ denotes the column space of the matrix $M$ and $M^*$ denotes the conjugate transpose of $M$. Notice that the condition $\mathcal{C}(N^*)\subseteq \mathcal{C}(M^*)$ can be replaced by $\mathcal{N}(N)\subseteq \mathcal{N}(M)$, where $\mathcal{N}(N)$ is the null space of the matrix $N$.

In \cite{Dragans}, Raki\'c and Djordjevi\'c extend the definition of space pre-order to the class of bounded linear operators on Banach spaces, and generalize some well known properties of this partial order to the new setting.

Observe that, whenever $A$ is unital, $a\leq_s b$ if and only if there exist $x,y\in A$ such that $a=bx=yb$.
It is easy to see that this relation is a partial order in every unital ring and that $a\leq_s b$ whenever $a\leq^{-} b$.

The following results are partially motivated by Theorem 2.5 and Theorem 3.7 in  \cite{Dragans}.

\begin{proposition}\label{space2} Let $A$ be a unital semiprime ring, $a\in A$ and $b\in A^\wedge$. The following conditions are equivalent:

\begin{enumerate}

\item $a\leq_s b$,

\item $\textrm{ann}_l(b)\subset \textrm{ann}_l(a)$ and $\textrm{ann}_r(b)\subset \textrm{ann}_r(a)$,

\item $a=bb^-a=ab^-b$ for every $b^-\in G_1(b)$,

\item $aD_1(b)a=\{0\}$.

\end{enumerate}

\end{proposition}

\begin{proof}It is clear that \emph{(1)}$\Rightarrow$\emph{(2)}.
In order to prove \emph{(2)}$\Rightarrow$\emph{(3)}, observe that $b(1-b^-b)=0$ for all $b^-\in G_1(b)$ and hence, by assumption, $a(1-b^-b)=0$ for every $b^-\in G_1(b)$. That is, $a=ab^-b$, for all $b^-\in G_1(b)$. Similarly, it can be proved that $a=bb^-a$ for every $b^-\in G_1(b)$.


Now suppose that \emph{(3)} holds and pick $x\in D_1(b)$. Then, $x=b^- - b^=$ for some $b^-,b^=\in G_1(b)$. By hypothesis we have $ab^-a=ab^-bb^=a=ab^=a$ and, consequently, $axa=ab^-a-ab^=a=0$. This proves that  \emph{(4)} holds.

Finally, assume that $aD_1(b)a=\{0\}$. By Lemma \ref{diff}, $a(1-b^-b)xa=0$ for all $x\in A$. Hence, $$a(1-b^-b)xa(1-b^-b)=0 \qquad (x\in A)$$ and, being $A$ semiprime, it yields $a=ab^-b$. Similarly, we obtain $a=bb^-a$ and therefore $a\leq_s b$.

\end{proof}

\begin{corollary}Let $A$ be a unital semiprime ring and $a,b\in A^\wedge$. Then $a\leq_s b$ if and only if $D_1(b)\subset D_1(a)$.

\end{corollary}
As a direct consequence of Proposition \ref{reg1} \emph{(3)}, and  \emph{(1)}$\Leftrightarrow$\emph{(3)} in Proposition \ref{space2},  we obtain the following characterization of the minus partial order.
\begin{corollary}Let $A$ be a unital semiprime ring and $a,b\in A^\wedge$. The following are equivalent:
\begin{enumerate}
\item $a\leq^{-} b$
\item $a\leq_s b$ and $G_1(a)\bigcap G_1(b)\neq \emptyset$.
\end{enumerate}
\end{corollary}

\medskip
\begin{definition}For a ring $A$ and $a,b\in A^\wedge$, we define $$G_1^b(a):=\{a^-\in G_1(a)\colon aa^-=ba^-, a^-a=a^-b\}.$$

\end{definition}

The following results provide algebraic adaptations for Theorems 3.8, 3.9 and 3.10 in \cite{Dragans}.

\begin{proposition}\label{tec1}Let $A$ be a ring and $a,b\in A^\wedge$ satisfying $a\leq^- b$. Then $$G_1^b(a)=\{b^--b^-(b-a)b^-\colon b^-\in G_1(b)\}.$$

\end{proposition}

\begin{proof} Since $a\leq^- b$, it follows from \cite[Lemma 2]{LebPaTh13} that $a=ab^-b=bb^-a=ab^-a$ for every $b^-\in G_1(b)$. Accordingly, an easy computation shows that, for every $b^-\in G_1(b)$,  $(b-a)b^-(b-a)=b-a$. In particular, $b-a\in A^\wedge$ and by Proposition \ref{minusreg}, $b-a\leq^- b$.

Let $a^-\in G_1^b(a)$ and $(b-a)^+\in G_2(b-a)$ such that $$(b-a)(b-a)^+=b(b-a)^+\quad\mbox{and}\quad (b-a)^+(b-a)=(b-a)^+b.$$ Then $$(b-a)a^-=0=a^-(b-a)\quad\mbox{and}\quad (b-a)^+a=0=a(b-a)^+.$$ Let $b^-=a^-+(b-a)^+$. From above it follows that $b^-\in G_1(b)$. Moreover,
\begin{equation}
b^--b^-(b-a)b^-= a^-+\left(b-a\right)^+ -\left(a^-+\left(b-a\right)^+\right)\left(b-a\right)\left(a^-+\left(b-a\right)^+\right)=a^-.
\end{equation}
Conversely, for every $b^{-}\in G_1(b)$ \begin{eqnarray*}
a\left(b^--b^-\left(b-a\right)b^-\right)&=&ab^--ab^-bb^-+ab^-ab=ab^-\\&=& ab^- =b\left(b^--b^-\left(b-a\right)b^-\right).\end{eqnarray*}
Similarly,
$$ \left(b^--b^-\left(b-a\right)b^-\right)a=b^-a=\left(b^--b^-\left(b-a\right)b^-\right)b.$$
Furthermore,
$$a\left(b^--b^-\left(b-a\right)b^-\right)a=ab^-a-\left(ab^-b\right)b^-a+\left(ab^-a\right)b^-a=a.$$
Therefore, $b^--b^-(b-a)b^-\in G_1^b(a)$, as desired.

\end{proof}

\begin{proposition}Let $A$ be a ring and $a,b\in A^\wedge$ such that $a\leq^-b$. The following assertions hold:

\begin{enumerate}

\item For every $a^-\in G_1^b(a)$, there exists $b^-\in G_1(b)$ satisfying $b^-a=a^-a$ and $ab^-=aa^-$,

\item For every $b^-\in G_1(b)$, there exists $a^-\in G_1^b(a)$ satisfying $b^-a=a^-a$ and $ab^-=aa^-$.

\end{enumerate}

\end{proposition}

\begin{proof}

In order to prove \emph{(1)}, pick $a^-\in G_1^b(a)$. By Proposition \ref{tec1} there is $b^-\in G_1(b)$ such that $$a^-=b^--b^-(b-a)b^-.$$ Hence, $$a^-a=(b^--b^-(b-a)b^-)a=b^-a-b^-bb^-a+b^-ab^-a=b^-a.$$ Similarly, $aa^-=ab^-$.

Now we prove \emph{(2)}. Let $b^-\in G_1(b)$. Again by Proposition  \ref{tec1}, we know that $a^-=b^--b^-(b-a)b^-\in G_1^b(a)$. As $a\leq^- b$, it follows $$aa^-=a(b^--b^-(b-a)b^-)=ab^--ab^-bb^-+ab^-ab^-=ab^-.$$ The identity $b^-a=a^-a$ can be obtained in the same way.

\end{proof}

\begin{proposition}Let $A$ be a unital complex algebra, $a,b\in A^\wedge$ such that $a\leq^- b$ and $c_1,c_2\in \mathbb{C}$ with $c_2\neq0$ and $c_1+c_2\neq 0$. Then $c_1a+c_2b\in A^{-1}$ if and only if $b\in A^{-1}$. Moreover, in such case $$(c_1a+c_2b)^{-1}=c_2^{-1}b^{-1}+((c_1+c_2)^{-1}-c_2^{-1})b^{-1}ab^{-1}.$$

\end{proposition}

\begin{proof}Suppose that $b\in A^{-1}$. As $a\leq^-b$, by the previous proposition, we have $G_1^b(a)=\{b^{-1}ab^{-1}\}$. In particular, this implies that $ab^{-1}ab^{-1}=ab^{-1}$ and $b^{-1}ab^{-1}a=b^{-1}a$. Now, by a direct computation \begin{multline*}\left(c_1a+c_2b\right)\left(c_2^{-1}b^{-1}+\left(\left(c_1+c_2\right)^{-1}-c_2^{-1}\right)b^{-1}ab^{-1}\right)=\\
													 =c_1c_2^{-1}ab^{-1}+c_1\left(\left(c_1+c_2\right)^{-1}-c_2^{-1}\right)ab^{-1}ab^{-1}+1+\\ c_2\left(\left(c_1+c_2\right)^{-1}-c_2^{-1}\right)ab^{-1}=\\
													 =1+\left(c_1c_2^{-1}+c_1\left(c_1+c_2\right)^{-1}-c_1c_2^{-1}+c_2\left(c_1+c_2\right)^{-1}-1\right)ab^{-1}=1.
\end{multline*}

Similarly, $$\left(c_2^{-1}b^{-1}+((c_1+c_2)^{-1}-c_2^{-1})b^{-1}ab^{-1}\right) \left( c_1a+c_2b \right) =1.$$
Conversely, if $c_1a+c_2b\in A^{-1}$, as $a=ab^-b=bb^-a$ for every $b^-\in G_1(b)$, we get $$c_1a+c_2b=(c_1ab^-+c_2)b=b(c_1b^-a+c_2).$$ Hence, $b$ is (left and right) invertible.

\end{proof}

\subsection{Minimal and maximal elements in the minus partial order}

Recall that, as we have proved in Corollary \ref{minuspartord}, the relation "$\leq ^{-}$" is a partial order on the set of regular elements of every unital semiprime ring.

Our next goal is describing the maximal and minimal elements of the minus partial order.





For a unital prime ring $A$, let us denote by  $A^{-1}_l$ and $A^{-1}_r$, the set of all left and right invertible elements of $A$, respectively.

\begin{proposition}\label{max} Let $A$ be a unital prime ring. The following conditions are equivalent:

\begin{enumerate}

\item $a\in A^\wedge$ and $a$ is maximal with respect to the relation "$\leq^-$",

\item $a\in A^{-1}_l\cup A^{-1}_r$.

\end{enumerate}

\end{proposition}

\begin{proof}First, given $a\in A^\wedge$ and $a^- \in G_1(a)$, it is easy to see that $$a\leq^- a+(1-aa^-)x(1-a^-a),$$ for every $x\in A$. If we suppose that $a$ is maximal, we get $(1-aa^-)x(1-a^-a)=0$, for every $x\in A$. As $A$ is a prime algebra, it yields $1=aa^-$ or $1=a^-a$.

Reciprocally, we may assume without loss of generality that $a$ is left invertible in $A$. If $a\leq^- b$, there exists $q\in A^\bullet$ such that $a=aq=bq$. Since $a \in A^{-1}_l$ it is clear that $q=1$ and hence, $a=b$. This shows that $a$ is maximal with respect to the relation "$\leq^-$".

\end{proof}

\begin{remark}Note that the condition of primality cannot be dropped in order to characterize maximal regular elements as left or right invertible elements. For instance, take $A=B(X)\oplus B(X)$ for an infinite dimensional Banach space $X$. This algebra is not prime but it is, in fact, semiprime. Take operators $L,R\in B(X)$ which are, respectively, left invertible and right invertible but none of them are invertible. The element $L\oplus R\in A$ is clearly maximal with respect to $\leq^-$ but it is neither left nor right invertible.

\end{remark}
Let $A$ be a unital semisimple Banach
algebra. For any element $a\in A$ denote by $\sigma (a)$ its
spectrum and by $\rm{r}(a)$ its spectral radius. The \emph{socle} of
$A$, $\soc (A)$, is the sum of all minimal left ideals of $A$, and
coincides with the sum of all minimal right ideals of $A$. Recall
that every minimal left ideal of $A$ is of the form $Ae$ for some
minimal idempotent $e$, that is, $e^2=e\neq 0$ with $eAe=\CC e$. If
$A$ has no minimal one-sided ideals, we let $\soc (A)=\{0\}.$

A non-zero element $u\in A$ is said of \emph{rank-one} if $u$
belongs to some minimal left ideal of $A$, that is, if $u=ue$ for
some minimal idempotent $e$ of $A$. It is known that $u$ has
rank-one if and only if $uau=\CC u\neq 0$, and that this is
equivalent to the condition $u\neq 0$ and $|\sigma(xu)|\setminus
\{0\}\leq 1, $ for all $x\in A$ (also equivalent to
$|\sigma(ux)|\setminus \{0\}\leq 1, $ for all $x\in A$).
For every rank-one
element $u$ in $A$, there exists $\tau(u)\in \CC$, such that $u^2=\tau(u)u$. Moreover, $\tau(u)=0$
or $\tau(u)$ is the only non-zero point of the spectrum of $u$. Thus, if $\tau(u)\neq 0$, then $\tau(u)^{-1}u$ is a minimal idempotent and $u=\tau(u)(\tau(u)^{-1}u)$.
Let us denote by $F_1(A)$ the set of rank-one elements of $A$.

Every element of the socle is a finite sum of
rank-one elements, that is to say that the socle coincides with the
set of all finite rank elements (see for instance
\cite{BreSe98}). Moreover, it is well known that $\soc (A)$ consists of regular elements.

Recall that a non-zero ideal $I$ of $A$ is called \emph{essential} if
it has non-zero intersection with every non-zero ideal of $A$. For a
semisimple Banach algebra $A$ this is equivalent to the condition
$aI=0$, for $a\in A$, implies $a=0$.

\begin{proposition}\label{minimal} Let $A$ be a unital semisimple Banach algebra with essential socle. Then, for every nonzero $a\in A$, there exists $u\in F_1(A)$ such that $u\leq^- a$. Furthermore, $u\in F_1(A)$ if and only if for every $v\leq^- u$ we have $u=v$ or $v=0$. In other words, the elements in $F_1(A)$ are precisely the nonzero minimal elements with respect to "$\leq^-$".

\end{proposition}

\begin{proof}

Fix $a\in A\setminus\{0\}$. Since $A$ is semisimple and has essential socle, there exists $w\in F_1(A)$ such that $aw\neq 0$.  Given  $(aw)^-\in G_1(aw)$, set $v=w(aw)^-$. It is clear that $av$ is a minimal idempotent and, in particular, $u=ava\in F_1(A)$. We claim that $u\leq^- a$. Indeed, let $p=av$ and $q=va$. These are idempotent elements in $A$, such that $$pu=avava=pa\quad \mbox{and}\quad uq=avava=aq.$$ Moreover, since $u=pa=aq$ it can be easily checked that  $ann_l(p)=ann_l(u)$ and $ann_r(q)=ann_r(u)$. 

Now, let $u\in F_1(A)$ and $0\neq v\leq^- u$. By Proposition \ref{reg1}, $v\in A^\wedge$ and there exists $v^+\in G_2(v)$ such that $v^+u=v^+v$ and $uv^+=vv^+$. Hence $v=vv^+u=uv^+v$ and, multiplying by $v^+u$ on the right, we get $$v=(uv^+v)v^+u=uv^+u=\tau(uv^+)u=u.$$

\end{proof}

\begin{definition} Let $A$ be a unital semisimple Banach algebra with nonzero socle. For every $u\in F_1(A)$ we define $$L_u:=\{ux:x\in A\}\quad\quad\mbox{and}\quad\quad R_u:=\{xu: x\in A\}.$$

\end{definition}

\begin{remark} These definitions are the algebraic analogue to the ones given in \cite[Theorem 8]{Semrl10}: for a rank one operator $S=x\otimes y^*\in B(H)$, we have $L_x=L_S$ and $R_y=R_S$. Indeed, if $R\in L_S$, then $R=ST$ for some $T\in B(H)$. Consequently, $$R=S(T(\cdot))=<T(\cdot),y>x=<\cdot \, , T^*(y)>x=x\otimes (T^*(y))^*$$ and hence, $R\in L_x$.

Now let $w^*\in B(H)^*$ and $R=x\otimes w^*\in L_x$.  Take an arbitrary operator $U$ such that $U(y)=w$ and set $T=U^*$. Then it can be proved that  $R=ST$, that is,  $R\in L_S$.

The equality $R_y=R_S$ is proven similarly.

\end{remark}
Notice that, for every $u\in F_1(A)$, $L_u=uA$ and $R_u=Au$ are  the right minimal ideal and the left minimal ideal generated by $u$, respectively. It is also clear that, for every $u\in F_1(A)$, $L_u$ and $R_u$ are subspaces of $\soc(A)$ consisting of elements of rank at most one.
Moreover, if $0\neq v\in L_u$, then $L_u=L_v$. Indeed, if $v=ux$ for some $x\in A$, then $$v(ux)^-u=ux(ux)^-u=\tau(ux(ux)^-)u=u,$$ which gives $u\in L_v$.( That is, $w=va$ for some $a\in A$ if and only if $w=ub$ for some $b\in A$). Similarly, if $0\neq v\in R_u$, then $R_u=R_v$.
\begin{lemma}\label{LRMax} The maximal linear subspaces of  $\soc(A)$ consisting of elements with rank at most one are precisely $L_u$ or $R_u$ where $u\in F_1(A)$.

\end{lemma}

\begin{proof}As we have just mentioned, for every $u\in F_1(A)$, $L_u$ and $R_u$ are linear subspaces of $\soc(A)$.

Let $u,v$ be non-zero elements such that $u,v,u+v\in F_1(A)$. For every $x\in A$, we have $(u+v)x(u+v)=\tau((u+v)x)(u+v)$, and by the additivity properties of the trace (see \cite[Lemma 4.3]{Puhl}) we know that $\tau((u+v)x)=\tau(ux)+\tau(vx)$. Hence
$$(u+v)x(u+v)=\left(\tau(ux)+\tau(vx)\right)(u+v),$$ which implies
 $$uxv+vxu=\tau(ux)v+\tau(vx)u.$$ Equivalently, $$(ux-\tau(ux))v=(\tau(vx)-vx)u,$$ for every $x\in A$.

 Assume that $z_0=(ux_0-\tau(ux_0))v=(\tau(vx_0)-vx_0)u$ is nonzero for some $x_0\in A$. Then we can write
 $$v=v z_0^{-}z_0=v((ux_0-\tau(ux_0))v)^-(\tau(vx_0)-vx_0)u,$$
 for $ z_0^{-}\in G_1(z_0)$.  This yields $v\in R_u$ and, consequently $R_u=R_v$.

 Otherwise, we have $uxv=\tau(ux)v$ for all $x\in A$. Therefore, $uu^-v=\tau(uu^-)v=v$ for any $u^-\in G_1(u)$. Thus, $v\in L_u$, which finally gives $L_u=L_v$. This shows that, for every linear subspace $M$ of $\soc(A)$ with $M\subset F_1(A)\cup\{0\}$ and $0\neq u \in M$, we have $M\subset L_u\cup R_u$ and hence, $M\subset L_u$ or $M\subset R_u$.

\end{proof}


\begin{proposition}\label{invert}Let $A$ be a unital semisimple Banach algebra with essential socle and $a\in A$. The following conditions are equivalent:

\begin{enumerate}

\item $a\in A^{-1}$,

\item $a\in A^\wedge$ and for every $u\in F_1(A)$, there exist $x\in L_u\setminus\{0\}$ and $y\in R_u\setminus\{0\}$ such that $x,y\leq^- a$.

\end{enumerate}

\end{proposition}

\begin{proof}

Let $a\in A^{-1}$, $u\in F_1(A)$ and $u^-\in G_1(u)$. It is clear that $x=uu^-a\in L_u\setminus\{0\}$. Let us show that $x\leq^- a$. Set $p=uu^-$ and $q=a^{-1}uu^-a$. Then $p,q\in A^\bullet$, $$px=uu^-uu^-a=uu^-a=pa$$ and $$xq=uu^-aa^{-1}uu^-a=uu^-a=aq.$$ Besides, it can be easily checked that $$\textrm{ann}_l(x)=\textrm{ann}_l(p) \quad\mbox{and}\quad \textrm{ann}_r(x)=\textrm{ann}_r(q).$$
Thus, $x \leq ^- a$. The existence of $y\in R_u\setminus\{0\}$ satisfying $y\leq^- a$ is guaranteed in the same way. This shows that \emph{(1)}$\Rightarrow$\emph{(2)}.

Conversely, let $a\in A$ satisfying \emph{(2)}. Given $u\in F_1(A)$, there exists $x\in A$ such that $ux\leq^- a$. As $ux$ is regular, there exists $(ux)^-\in G_1(ux)$ such that $(ux)^-(ux)=(ux)^-a$ and $(ux)(ux)^-=a(ux)^-$. Multiplying the last identity by $u$ on the right, it yields $(ux)(ux)^-u=a(ux)^-u$. As $u\in F_1(A)$, we have $(ux)(ux)^-u=\tau((ux)(ux)^-)u=u$ and, hence, $u=a(ux)^-u$. Similarly, given $y\in A$ such that $yu\leq^- a$, we get $u=u(yu)^-a$.

Suppose that $za=0$. In such case $za(ux)^-u=zu=0$, for every $u\in F_1(A)$. Since $A$ is semisimple and has essential socle, it gives $z=0$.  We have proved that $\textrm{ann}_l(a)=\{0\}$. Similarly, it can be checked that $\textrm{ann}_r(a)=\{0\}$. Therefore, $a$ is a regular element which is not a zero divisor, that is, $a$ is invertible.

\end{proof}
\section{Maps preserving the minus partial order}
Let $A$ and $B$ be Banach algebras. A linear map $\Phi:A\to B$  is a \emph{Jordan homomorphism} if $\Phi(a^2)=\Phi(a)^2$, for all $a\in A$, equivalently $\Phi(a\circ b)=\Phi(a)\circ \Phi(b)$, for every $a,b\in A$, where $\circ$ denotes the usual Jordan product $a\circ b=\frac{1}{2}(ab+ba)$ . If $A$ and $B$ are unital, $\Phi$ is called \emph{unital} if $\Phi(1)=1$, where $1$ is used for the identity element of both $A$ and $B$.
Clearly every homomorphism and every anti-homomorphism is a Jordan homomorphism.  A well known result of Herstein, \cite{Her56}, states that every surjective Jordan homomorphism $T:A \to B$ is either an homomorphism or an
anti-homomorphism whenever $B$ is prime.

It is also well known that if $\Phi:A \to B$ is a Jordan homomorphism, then $\Phi$ is a \emph{Jordan triple homomorphism}, that is, $$ \Phi\left(\{a, b, c\}\right)=\{\Phi(a),\Phi(b),\Phi(c)\}, \quad\mbox{for all }a,b,c\in A,$$ where $\{a, b, c\}=\frac{1}{2}(abc+cba)$ is the usual Jordan triple product in $A$. In particular, it is clear that every Jordan (triple) homomorphism,  $\Phi:A\to B$, strongly preserves regularity, that is, if $a^{-}\in G_1(a)$ then  $\Phi(a)^{-}\in G_1(\Phi(a))$. Obviously, if $a^{-}\in G_2(a)$ then  $\Phi(a)^{-}\in G_2(\Phi(a))$ and hence $\Phi(A^\wedge)\subseteq B^\wedge$.

\begin{proposition}\label{triple}Let $A,B$ be Banach algebras and $\Phi:A\to B$ a Jordan triple homomorphism. Then, $a\leq^-b$ implies $\Phi(a)\leq^- \Phi(b)$, for every $a,b\in A^\wedge$.

\end{proposition}


\begin{proof} Let $a,b\in A^\wedge$. By Proposition \ref{reg1} \emph{(3)}, $a\leq^-b$ if and only if there exists $b^-\in G_1(b)$ such that $a=ab^-a=ab^-b=bb^-a$. We may assume that $b^-\in G_2(b)$. Since $\Phi$ is a  Jordan triple homomorphism, and $a=ab^-a$ and $2a=ab^-b+bb^-a$, we have $$\Phi(a)=\Phi(a)\Phi(b)^-\Phi(a)\quad \mbox{ and}\quad 2\Phi(a)=\Phi(a)\Phi(b)^-\Phi(b)+\Phi(b)\Phi(b)^-\Phi(a).$$  Multiplying  the last identity by $\Phi(b)^-\Phi(a)$ on the right, and havind in mind that, as we have previously point out, $\Phi(b)^-\in G_2(\Phi(b))$, we get \begin{eqnarray*}2\Phi(a)&=& \Phi(a)\Phi(b)^-\Phi(b)\Phi(b)^-\Phi(a)+\Phi(b)\Phi(b)^-\Phi(a)\Phi(b)^-\Phi(a)\\ &=& \Phi(a)+\Phi(b)\Phi(b)^-\Phi(a).
\end{eqnarray*} Consequently, $\Phi(a)=\Phi(b)\Phi(b)^-\Phi(a)$. Similarly, it can be obtained that $\Phi(a)=\Phi(a)\Phi(b)^-\Phi(b)$, which completes the proof.
\end{proof}

Recall that every unital Jordan homomorphism between Banach algebras, $\Phi:A\to B$ \emph{preserves invertibility}, that is, $\Phi(a)\in B^{-1}$ for every $a\in A^{-1}$. (In fact, it
strongly preserves invertibility, that is, $\Phi(a^{-1})=\Phi(a)^{-1}$, for every invertible element $a\in A$.) In \cite{BreFosSem}, Bre\v sar, Fo\v sner and \v Semrl, showed that every unital bijective invertibility preserving
linear map, $\Phi :A \to B$, between semisimple Banach algebras, is a
Jordan isomorphism whenever $A$ has essential socle.

We present now  the main result in this section.




\begin{theorem}\label{main1} Let $A$ and $B$ be unital semisimple Banach algebras with essential socle. Let $\Phi:A\to B$ be a bijective linear map. The following conditions are equivalents:
\begin{enumerate}
\item $\Phi(A^\wedge)= B^\wedge$, and $a\leq^-b \Leftrightarrow \Phi(a)\leq^- \Phi(b)$, for every $a,b\in A^\wedge$.
\item $\Phi$ is a Jordan isomorphism multiplied by an invertible element.
\end{enumerate}
\end{theorem}
\begin{proof}It is clear from Proposition \ref{reg1} \emph{(4)} and Proposition \ref{triple} that \emph{(2)}$\Rightarrow$\emph{(1)}.

Suppose now that  $\Phi(A^\wedge)= B^\wedge$ (that is, $a$ is regular if and only if $\Phi(a)$ is regular) and  $$a\leq^-b \Leftrightarrow \Phi(a)\leq^- \Phi(b), \quad \mbox{ for every }a,b\in A^\wedge.$$

We will show that $\Phi(F_1(A))=F_1(B)$. Let $u\in F_1(A)$. Hence $\Phi(u)$ is a nonzero regular element, and by Proposition  \ref{minimal} there exists $\Phi(v)\in F_1(B)$ such that $\Phi(v)\leq ^- \Phi(u)$. From Proposition \ref{reg1} \emph{(2)}, $\Phi(v)\in B^\wedge$ and $G_1(\Phi(u))\subseteq G_1(\Phi(v))$. By hypothesis, $v\leq ^- u$. Since $u\in F_1(A)$ and $v\neq 0$, again by Proposition  \ref{minimal}, $v=u$. That is $\Phi(v)=\Phi(u)$, which shows that $\Phi(v)\in F_1(B)$. Taking into account that $\Phi^{-1}$ satisfies the same conditions, we get $\Phi(F_1(A))=F_1(B)$.

Now, since the sets $L_u$ and $R_u$ are the maximal linear subspaces of $\soc(A)$ consisting of elements with rank at most one (see Lemma \ref{LRMax}), we conclude that $\Phi(L_u), \Phi(R_u)\in \{L_{\Phi(u)}, R_{\Phi(u)}\}$ for every $u\in F_1(A)$.

Let $a\in A^{-1}$. We claim that $\Phi(a)\in B^{-1}$. By hypothesis,  $\Phi(a)\in B^{\wedge}$. Given $\Phi(u)\in F_1(B)$, since $a\in  A^{-1}$,  we know by Proposition  \ref{invert} that there exist $x_0\in L_u\setminus\{0\}$ and $y_0\in R_u\setminus\{0\}$ such that $x_0,y_0\leq^- a$. If $\Phi^{-1}(L_{\Phi(u)})=L_u$, take $x=x_0$. Otherwise, take $x=y_0$. Then $\Phi(x)\in L_{\Phi(u)}$ and $\Phi(x)\leq^{-} \Phi(a)$.
Similarly, we find $\Phi(y)\in R_{\Phi(u)}$ with $\Phi(y)\leq^- \Phi(a)$. This shows that $\Phi(a)\in B^{-1}$.

Let $\Psi:A\to B$ be the linear mapping given by $\Psi(x)=\Phi(1)^{-1}\Phi(x)$, for all $x\in A$. It is clear that $\Psi$ is unital, bijective and preserves invertibility. By \cite[Theorem 1.1]{BreFosSem} $\Psi$ is a Jordan isomorphism, which concludes the proof.
\end{proof}

\begin{remark}\label{rem-op} Let $X$ be a complex Banach space. Denote by $B(X)$ the algebra of all bounded linear operators on $X$. This is a unital semisimple Banach algebra with essential socle. Notice that $\soc(B(X))=F(X)$ is the ideal of finite rank operators on $X$.

Let $T\in B(X)$. By looking at the proof of (2)$\Rightarrow$(1) in Proposition \ref{invert} it can be seen that, if $T$ satisfies that, for every rank one operator  $U\in F_1(X)$, there exist $L\in L_U\setminus\{0\}$ and $R\in R_U\setminus\{0\}$ with $L,R\leq^- T$, then $T$ is invertible.
\end{remark}

Let $A$ be a unital semisimple Banach algebra with essential socle, and $X$ be a complex Banach space. Let $\Phi: A\to B(X)$ be a surjective linear map such that

$$a\leq^- b\quad \mbox{ if and only if} \quad \Phi(a)\leq^- \Phi(b).$$
Notice that $\Phi$ is injective: if $\Phi(x)=0$, then $\Phi(x)\leq^- \Phi(0)$, which by assumption, gives that $x\leq^- 0$, and finally $x=0$.

A direct application of Proposition \ref{minimal} shows that $\Phi(F_1(A))=F_1(B)$, and by  Lemma \ref{LRMax} $\Phi(L_u), \Phi(R_u)\in \{L_{\Phi(u)}, R_{\Phi(u)}\}$ for every $u\in F_1(A)$. From this facts and Remark \ref{rem-op}, it is clear now that $\Phi$ preserves invertibility. As in the previous theorem, it follows that the linear mapping given by $\Psi(x)=\Phi(1)^{-1}\Phi(x)$, is a Jordan isomorphism. By the Herstein's theorem (\cite{Her56}) $\Psi$ is either an isomorphism or an anti-isomorphism.
On the other hand, it is  straightforward to check that every  isomorphism  and every anti-isomorphism preserves the minus partial relation in both directions. In view of Proposition \ref{reg1} \emph{(4)}, this is also the case for every isomorphism or anti-isomorphism multiplied by an invertible element.

 This proves the next result.

\begin{theorem}\label{op}
Let $A$ be a unital semisimple Banach algebra with essential socle, and $X$ be a complex Banach space. Let $\Phi: A\to B(X)$ be a surjective linear map.
The following are equivalent:
\begin{enumerate}
\item $a\leq^- b$ if and only if $\Phi(a)\leq^- \Phi(b)$, for every $a,b\in A$.
\item $\Phi$ is either an isomorphism multiplied by an invertible element, or an anti-isomorphism multiplied by an invertible element.
\end{enumerate}

\end{theorem}

Let $A$ be a unital prime C*-algebra with non zero socle. Then $A$ is primitive and has essential socle.
Let us assume that $e$ is a minimal projection in $A$. Then
the minimal left ideal $ Ae$ can be endowed wit inner product, $\langle x, y\rangle e = y^*x$,
(for all $ x, y \in Ae$), under which $Ae$ becomes a Hilbert space in the algebra norm.
Let $\rho : A \to B(Ae)$ be the left regular representation on $Ae$, given by $\rho(a)(x) = ax$.
The mapping  $\rho$ is an isometric irreducible $^*$-representation, satisfying:
\begin{enumerate}
\item $\rho(\soc(A))=F(Ae)$,
\item $\rho(\overline{\soc(A)})=K(Ae)$,
\item $\sigma_A(x)=\sigma_{ B(Ae)}(\rho(x))$, for every $x\in A$.
\end{enumerate}
(See \cite[Section F.4]{Bar}.)
Let $a\in A$ (non necessarily regular) such that, for every $u\in F_1(A)$, there exist $x\in L_u\setminus\{0\}$ and $y\in R_u\setminus\{0\}$ such that $x,y\leq^- a$. As we have proved in Proposition \ref{invert}, given $u\in F_1(A)$, we can find $w,z\in A$ such that $u=awu=uza$, which in particular shows that $a$ is not a zero divisor.  We claim that $\rho(a)$ is invertible, which, in this setting, shows that $a$ is invertible. Indeed, since $\textrm{ann}_r(a)=\{0\}$ it is clear that $\rho(a)$ is injective. Moreover, given $ze\in Ae$, by hypothesis, there exist $w\in A$ such that $ze=awze=\rho(a)(wze)$. This shows that $\rho(a)$ is surjective, and hence $\rho(a)$ is invertible.
We have just proved the following:

\begin{proposition}\label{invertprime} Let $A$ be a unital prime C*-algebra with nonzero socle and $a\in A$. The following conditions are equivalent:

\begin{enumerate}

\item $a\in A^{-1}$,

\item For every $u\in F_1(A)$, there exist nonzero $x\in L_u$ and $y\in R_u$ such that $x,y\leq^- a$.

\end{enumerate}

\end{proposition}

The proof of the next theorem follows the lines of Theorems \ref{main1} and \ref{op}, by using Propositions \ref{invertprime}, \ref{minimal} and \cite[Theorem 1.1]{BreFosSem}.

\begin{theorem}\label{mainminus}Let $A$ be a unital semisimple Banach algebra with essential socle, $B$ a unital prime C*-algebra with nonzero socle and $\Phi: A\to B$ a surjective linear map.
The following are equivalent:
\begin{enumerate}
\item $a\leq^- b$ if and only if $\Phi(a)\leq^- \Phi(b)$, for every $a,b\in A$,
\item $\Phi$ is either an isomorphism multiplied by an invertible element, or an anti-isomorphism multiplied by an invertible element.
\end{enumerate}

\end{theorem}
We will like to shed some light on the study of mappings preserving the relation ''$\leq^{-}$'' just in one direction.

Recall that a
$C^*$-algebra $A$ is of \emph{real rank zero} if the set of all real
linear combinations of orthogonal projections is dense in the set of
all hermitian elements of $A$ (see \cite{BroPe91}). Notice that
every von Neumann algebra, and, in particular, the
algebra of all bounded linear operators on a complex Hilbert space $H$ is of real rank zero.

\begin{theorem}\label{rro} Let $A$ be a real rank zero $C^*$-algebra and $B$ be unital Banach algebra. Let $\Phi:A\to B$ be a bounded linear map satisfying that
$$a\leq^{-}b\quad\mbox{ implies}\quad \Phi(a)\leq^{-}\Phi(b),\quad\mbox{for all }a,b\in A.$$
The following assertions hold:
\begin{enumerate}
\item If $\Phi(1)\in B^\bullet$ then $\Phi$ is a Jordan homomorphism,
\item If $\Phi(A)\cap B^{-1}$ and $\Phi(1)\in B^\wedge$ then $\Phi$ is a Jordan homomorphism multiplied by an invertible element.
\end{enumerate}
\end{theorem}
\begin{proof}

\emph{(1)} Assume that $\Phi(1)\in B^\bullet$. For every $p\in A^\bullet$, as $p\leq^- 1$ it follows that  $\Phi(p)\leq^- \Phi(1)$. Having in mind Proposition \ref{reg1}\emph{(5)}, we conclude that $\Phi(p)\in B^\bullet $ for all $p\in A^\bullet$.

This shows that $\Phi$ preserves idempotents. It is well known that in this case if $p$ and $q$ are mutually orthogonal projections then $\Phi(p)$ and $\Phi(q)$ are mutually orthogonal idempotents in $B$, and hence $\Phi$ is a Jordan homomorphism (see, for instance, \cite[Lemma 3.1]{Ko05}).

\emph{(2)} Suppose that $\Phi(A)\cap B^{-1}$ and $\Phi(1)\in B^\wedge$. As above, $\Phi(e)\leq^- \Phi(1)$  for every $e\in A^\bullet$. In particular, $\Phi(e)\in \Phi(1)B\cap B\Phi(1)$  for every $e\in A^\bullet$. Since $\Phi(1)$ is regular, it is well known that $\Phi(1)B$ and $B\Phi(1)$ are closed. Taking into account that every
self-adjoint element in $A$ can be approximated by linear combinations of mutually orthogonal projections, and that $\Phi$ is linear and bounded, we conclude that $\Phi(x)\in \Phi(1)B\cap B\Phi(1)$  for every $x\in A$. Therefore, as $\Phi(A)\cap B^{-1}$, we deduce that $\Phi(1)$ is invertible.
Finally, let $\Psi:A\to B$ be the linear mapping defined as $\Psi(x)=\Phi(1)^{-1}\Phi(x)$, for all $x\in A$. We conclude the proof by proving that $\Psi$ preserves idempotents: given $e\in A^\bullet$, since $\Phi(e)\leq^- \Phi(1)$, there exists $p\in B^\bullet$ such that $\Phi(e)=\Phi(1)p$, that is $\Psi(e)=\Phi(1)^{-1}\Phi(e)=p\in B^\bullet$.
\end{proof}

We conclude this paper with two remarks. The first one shows that it is not always possible to characterize Jordan homomorphims in terms of minus partial order preserving conditions. The second one deals with linear maps preserving the space preorder (see Definition \ref{space}).
\begin{remark}
Let $A$ be a Rickart ring. By \cite[Theorem 3.3]{DjoRaMa13}, the relation ''$\leq^-$'' defines a partial order in $A$. Every linear mapping $\Phi:\mathbb{C}\to A$ preserves the minus partial order. Notice that $a\leq^-b$ in $\mathbb{C}$ if and only if $a=0$ or $a=b$, and that, by reflexivity, $\Phi(a)\leq^-\Phi(a)$ for every $a\in \mathbb{C}$.

Observe that the same conclusions hold when $\mathbb{C}$ is replaced by any Banach algebra $A$ in which the only idempotents are the trivial ones, namely, the identity and zero. For instance $A=C([0,1])$.
\end{remark}
\begin{remark}
Let $A$ and $B$ be unital semisimple Banach algebras. Let $\Phi:A\to B$ be a linear mapping such that
$$ a\leq_s b \quad \mbox{implies}\quad \Phi(a)\leq_s \Phi(b), \quad\mbox{for all}\quad  a,b\in A.$$
Notice that $b\in A^{-1}$ if and only if $a\leq_s b$ for every $a\in A$. Hence, if $\Phi$ is surjective then $\Phi$ preserves invertibility.
This shows that $\Phi$ is a Jordan homomorphism multiplied by an invertible element in the following settings:
\begin{enumerate}
\item If $A$ has essential socle (\cite[Theorem 1.1]{BreFosSem})
\item If $A$ has real rank zero (\cite[Theorem 3.1]{CuHo04})
\end{enumerate}

\end{remark}


\end{document}